\documentclass[11pt]{amsart}

\scrollmode

\usepackage{times}
\usepackage{amsmath,graphicx}
\usepackage{latexsym}
\usepackage{amscd,amsthm,amssymb}
\usepackage[all]{xy}

\newtheorem{thm}{Theorem}
\newtheorem{prop}[thm]{Proposition}
\newtheorem{lem}[thm]{Lemma}
\newtheorem{cor}[thm]{Corollary}

\newtheoremstyle{CorNum}
        {\topsep}{\topsep}              
        {\itshape}                      
        {}                              
        {\bfseries}                     
        {.}                             
        { }                             
        {\thmname{#1}\thmnote{ \bfseries #3}}
    \theoremstyle{CorNum}
    \newtheorem{corn}{Corollary}

\theoremstyle{definition}
\newtheorem{ex}[thm]{Example}
\newtheorem{rem}[thm]{Remark}
\newtheorem{defn}[thm]{Definition}

\title{Inequivalent contact structures on Boothby-Wang 5-manifolds}
\author{M.~J.~D.~Hamilton}
\address{      Universit\"at Stuttgart\\
               Fachbereich Mathematik\\
               Pfaffenwaldring 57\\
               70569 Stuttgart\\
               Germany}
\email{mark.hamilton@math.lmu.de}
\date{\today}

\begin{document}

\begin{abstract}We consider contact structures on simply-connected 5-manifolds which arise as circle bundles over simply-connected symplectic 4-manifolds and show that invariants from contact homology are related to the divisibility of the canonical class of the symplectic structure. As an application we find new examples of inequivalent contact structures in the same equivalence class of almost contact structures with non-zero first Chern class.
\end{abstract}

\maketitle

\section{Introduction}
Suppose that $(M,\omega)$ is a symplectic manifold so that the symplectic form $\omega$ represents an integral class $[\omega]$ in $H^2(M;\mathbb{R})$. We can consider the circle bundle $X$ over $M$ whose Euler class is equal to the class $[\omega]$. The Boothby-Wang construction \cite{BoWa} associates to each symplectic manifold $(M,\omega)$ with an integral symplectic class a contact structure $\xi$ on the manifold $X$. In this article we are interested in the case where $X$ is a simply-connected closed 5-manifold. In Section \ref{sect topology circle bundles} we will show that in this case the 4-manifold $M$ also has to be simply-connected and the Euler class $[\omega]$ is indivisible. In addition, it follows that the integral homology of $X$ is torsion free. By the classification of simply-connected closed 5-manifolds due to D.~Barden \cite{Bar} it is possible to determine the 5-manifold $X$ up to diffeomorphism: $X$ is diffeomorphic either to the connected sum
\begin{equation*}
\#(b_2(M)-1)S^2\times S^3
\end{equation*}
or to
\begin{equation*}
\#(b_2(M)-2)S^2\times S^3\#S^2\tilde{\times}S^3,
\end{equation*}
depending on whether $X$ is spin or non-spin. Here $S^2\tilde{\times}S^3$ denotes the non-trivial $S^3$-bundle over $S^2$. Moreover, the 5-manifold $X$ is spin if and only if $M$ is spin or the mod $2$ reduction of the Euler class $[\omega]$ is equal to the second Stiefel-Whitney class of $M$. As a consequence of this diffeomorphism classification one can construct Boothby-Wang contact structures on the same simply-connected 5-manifold $X$ using different simply-connected symplectic 4-manifolds $(M,\omega)$ and $(M',\omega')$. Up to the spin condition the 4-manifolds only need to have the same second Betti number. 

In Section \ref{sect contact structures} we consider contact structures and almost contact structures on simply-connected 5-manifolds in general. In particular, we consider the notion of {\em equivalence} of these structures, i.e.~when two such structures can be made identical by a sequence of deformations and self-diffeomorphisms of the manifold. We will show that two almost contact structures are equivalent on a simply-connected 5-manifold if and only if their first Chern classes have the same maximal divisibility. We call this divisibility the {\em level} of the (almost) contact structure. Hence contact structures on the same level are equivalent as almost contact structures.

Since symplectic 4-manifolds exist in great number it is likely that many of the induced Boothby-Wang contact structures on the same 5-manifold $X$ are not equivalent as contact structures, even if they are equivalent as almost contact structures. In Section \ref{sect construct 4-mfds} and \ref{sect contact homology} we will show that invariants derived from contact homology defined in \cite{EGH} are related to the divisibility of the canonical class of the symplectic structure on the simply-connected 4-manifold. This is summarized in the following main result:
\vspace{0.2cm}
\begin{corn}[\ref{cor different contact str}] Let $X$ be a closed simply-connected 5-manifold which can be realized in two different ways as a Boothby-Wang fibration over closed simply-connected symplectic 4-manifolds $(M_1,\omega_1)$ and $(M_2,\omega_2)$, whose symplectic forms represent integral and indivisible classes:
\begin{equation*}
\xymatrix{&X \ar[dl]_{\pi_1}\ar[dr]^{\pi_2}&\\
          (M_1,\omega_1)&&(M_2,\omega_2)}
\end{equation*}
Denote the associated Boothby-Wang contact structures on $X$ by $\xi_1$ and $\xi_2$ and the canonical classes of the symplectic structures by $K_1$ and $K_2$. Let $d(\xi_i)$ denote the divisibility of the first Chern class of $\xi_i$ and $d(K_i)$ the divisibility of $K_i$. Then:
\begin{itemize}
\item The almost contact structures underlying $\xi_1$ and $\xi_2$ are equivalent if and only if $d(\xi_1)=d(\xi_2)$.
\end{itemize}
Suppose that $\xi_1$ and $\xi_2$ are equivalent as contact structures.
\begin{itemize}
\item If $d(\xi_1)=d(\xi_2)=0$, then $d(K_1)=d(K_2)$.
\item If $d(\xi_1)=d(\xi_2)\neq 0$, then either both $d(K_1),d(K_2)\leq 3$ or $d(K_1)=d(K_2)\geq 4$.
\end{itemize}
\end{corn}
\vspace{0.1cm}
Hence the existence of inequivalent contact structures on simply-connected 5-manifolds with torsion free homology is connected to the geography question of simply-connected 4-manifolds with divisible canonical class. As an application we find in Section \ref{sect applications} new examples of inequivalent contact structures in the same equivalence class of almost contact structures with non-zero first Chern class. To state a result we consider the following purely number theoretic definition:  Let $d\geq 4$ be an integer. Consider the number of divisors greater or equal to four of $d$. Then $N(d)$ is this number plus one. If $d$ is even, consider the number of odd divisors greater or equal to four of $d$. Then $N'(d)$ is this number plus one. Using geography results for symplectic homotopy elliptic surfaces we get:
\vspace{0.2cm}
\begin{corn}[\ref{cor on existence of contact}] Let $n\geq 1$ be an arbitrary integer.
\begin{enumerate}
\item On every odd level $d\geq 5$ the 5-manifold $\#(12n-4)S^2\times S^3\#S^2\tilde{\times}S^3$ admits at least $N(d)$ inequivalent contact structures.
\item On every even level $d\geq 4$ the 5-manifold $\#(24n-3)S^2\times S^3$ admits at least $N(d)$ inequivalent contact structures.
\item On every even level $d\geq 4$ the 5-manifold $\#(24n-15)S^2\times S^3$ admits at least $N'(d)$ inequivalent contact structures.
\end{enumerate}
\end{corn}
\vspace{0.1cm}
Hence as the level increases we get many inequivalent contact structures with non-zero first Chern class. A related discussion has appeared in \cite{P}. Inequivalent contact structures on simply-connected 5-manifolds with vanishing first Chern class have been found before by O.~van Koert in \cite{Koe1}. Also I.~Ustilovsky \cite{U} found infinitely many contact structures on the sphere $S^5$ and F.~Bourgeois \cite{Bou} on $T^2\times S^3$ and $T^5$, both in the case of vanishing first Chern class.  

\section{Classification of simply-connected 5-manifolds}

Throughout this article we use for a topological space $Y$ the abbreviations $H_*(Y)$ and $H^*(Y)$ to denote the homology and cohomology groups of $Y$ with $\mathbb{Z}$-coefficients. Other coefficients will be denoted explicitly. 

In this section we recall the classification of simply-connected closed 5-manifolds due to D.~Barden \cite{Bar} and refer to this article for further details. Let $X$ denote a smooth closed oriented 5-manifold. For each pair of elements $\eta,\xi$ in the torsion subgroup $\text{Tor}\,H_2(X)$ there exists a {\em linking number} $b(\eta,\xi)$ in $\mathbb{Q}/\mathbb{Z}$. These numbers define a skew-symmetric non-degenerate bilinear form
\begin{equation*}
b\colon \text{Tor}\,H_2(X)\times \text{Tor}\,H_2(X)\longrightarrow \mathbb{Q}/\mathbb{Z},
\end{equation*}
called the {\em linking form}. Suppose that the 5-manifold $X$ is simply-connected. Then the first integral homology group vanishes and the universal coefficient theorem implies that there exists an isomorphism 
\begin{equation*}
H^2(X;\mathbb{Z}_2)\cong \text{Hom}(H_2(X),\mathbb{Z}_2),
\end{equation*}
via evaluation of cohomology on homology classes. Hence we can think of the second Stiefel-Whitney class $w_2(X)\in H^2(X;\mathbb{Z}_2)$ as a homomorphism
\begin{equation*}
w_2(X)\colon H_2(X)\longrightarrow \mathbb{Z}_2.
\end{equation*} 
The following theorem is the classification theorem for simply-connected 5-manifolds and was proved by Barden \cite[Theorem 2.2]{Bar} using surgery theory:
\begin{thm}\label{Barden} Let $X,Y$ be simply-connected closed oriented 5-manifolds. Suppose that $\theta\colon H_2(X)\rightarrow H_2(Y)$ is an isomorphism preserving the linking forms on the torsion subgroups and such that $w_2(Y)\circ\theta=w_2(X)$. Then there exists an orientation preserving diffeomorphism $f\colon X\rightarrow Y$ such that $f_*=\theta$. 
\end{thm}
Since the linking number and the second Stiefel-Whitney class are homotopy invariants, it follows in particular that simply-connected closed 5-manifolds which are homotopy equivalent are already diffeomorphic.

It is possible to give a complete list of building blocks of simply-connected 5-manifolds such that each simply-connected 5-manifold is a connected sum of some of those building blocks. In the following we are particularly interested in simply-connected 5-manifolds $X$ whose integral homology is torsion free. By Poincar\'e duality and the universal coefficient theorem the whole integral homology is torsion free if and only if the second homology $H_2(X)$ is torsion free. Simply-connected 5-manifolds satisfying this condition have a simple structure, because they can be constructed using only two building blocks, which can be described in the following way.

There exist up to isomorphism precisely two oriented $S^3$-bundles over $S^2$ -- the trivial bundle $S^2\times S^3$ and a non-trivial bundle, denoted by $S^2\tilde{\times}S^3$. The manifold $S^2\tilde{\times}S^3$ can be constructed as follows: Let $B=S^2\tilde{\times}D^3$ denote the non-trivial $D^3$-bundle over $S^2$. Then the boundary $\partial B$ is the non-trivial $S^2$-bundle over $S^2$, hence diffeomorphic to $\mathbb{CP}^2\#{\overline{\mathbb{CP}}{}^{2}}$. Consider a second copy $B^*$ of $B$ with the opposite orientation. Then the oriented 5-manifold $S^2\tilde{\times}S^3$ is obtained by gluing together $B$ and $B^*$ along their boundaries via the identity. In particular, the manifold  $S^2\tilde{\times}S^3$ is non-spin, because a spin structure would induce a spin structure on $B$ and hence on $\partial B$, which is non-spin.  

It follows from the list of building blocks in Barden's article \cite{Bar} that $S^2\times S^3$ and $S^2\tilde{\times}S^3$ are the only building blocks with torsion free second integral homology. Hence every simply-connected 5-manifold with torsion free homology decomposes as a connected sum of several copies of these two manifolds. Moreover, one can show with Theorem \ref{Barden} that there exists a diffeomorphism
\begin{equation*}
S^2\tilde{\times}S^3\# S^2\tilde{\times}S^3 \cong S^2\times S^3\# S^2\tilde{\times}S^3,
\end{equation*}
hence in every non-spin connected sum one $S^2\tilde{\times}S^3$ summand suffices. This implies:
\begin{prop}\label{5-mfd tor free H2}
Let $X$ be a simply-connected closed oriented 5-manifold with torsion free homology. Then $X$ is diffeomorphic to
\begin{enumerate}
\item $\#b_2(X)S^2\times S^3$ if $X$ is spin
\item $\#(b_2(X)-1)S^2\times S^3\#S^2\tilde{\times}S^3$ if $X$ is non-spin.
\end{enumerate}
\end{prop}
Here we denote by $\#rS^2\times S^3$ the connected sum
\begin{equation*}
S^2\times S^3\#S^2\times S^3\#\ldots\#S^2\times S^3
\end{equation*}
of $r$ copies of $S^2\times S^3$. The empty sum in (a) for $b_2(X)=0$ is the 5-sphere $S^5$.

\section{Contact structures on simply-connected 5-manifolds}\label{sect contact structures}

Let $X^{2n+1}$ denote a connected oriented manifold of odd dimension. By definition, an {\em almost contact structure} on $X$ is a rank $2n$-distribution $\xi\subset TX$ together with a symplectic structure $\sigma$ on the vector bundle $\xi\rightarrow X$. A (co-orientable) {\em contact structure} is an almost contact structure such that the symplectic form $\sigma$ on $\xi$ is of the form $(d\alpha)|_\xi$, where $\alpha$ is a nowhere vanishing 1-form on $X$ that defines $\xi$ in the sense that the kernel distribution $\text{ker}\,\alpha$ equals $\xi$. Note that there is a slightly more general version of contact structures which are not co-orientable so that the defining 1-form and hence the symplectic structure on the distribution exists only locally on $X$. In the following we only consider co-orientable contact structures.

If $(\xi,\sigma)$ is an almost contact structure we can choose a complex structure on $\xi$ compatible with the symplectic form $\sigma$ and hence define Chern classes $c_k(\xi)\in H^{2k}(X)$. These classes do not depend on the choice of compatible complex structure, because the space of complex structures compatible with a given symplectic form is contractible. However, they depend on the choice of symplectic structure. For a contact structure we can choose complex structures compatible with the symplectic form $(d\alpha)|_\xi$ for a defining 1-form $\alpha$. Since any two defining 1-forms only differ by multiplication with a nowhere zero function on $X$, it follows that the Chern classes $c_k(\xi)$ of a contact structure depend only on the contact distribution $\xi$, not on the choice of contact form $\alpha$. 

The first Chern class of an almost contact structure $\xi$ is related to the second Stiefel-Whitney class of the manifold $X$ in the following way: 
\begin{lem}\label{Contact spin} Let $\xi$ be an almost contact structure on $X$. Then $c_1(\xi)\equiv w_2(X) \bmod 2$. 
\end{lem}
\begin{proof} By the Whitney sum formula for $TX=\xi\oplus\mathbb{R}$,
\begin{equation*}
w_2(X)=w_2(\xi)\cup w_0(\mathbb{R})=w_2(\xi).
\end{equation*} 
Since $\xi\rightarrow X$ is a complex vector bundle, with complex structure compatible with $\sigma$, we have $w_2(\xi)\equiv c_1(\xi) \bmod 2$. This implies the claim.
\qed
\end{proof}

Suppose that $\xi_t$ for $t\in [0,1]$ is a smooth family of contact structures on a closed manifold $X$. We can choose a smooth family of $1$-forms $\alpha_t$ defining $\xi_t$. Using the Moser technique one can prove that there exists a smooth family $\psi_t$ of orientation preserving self-diffeomorphisms of $X$ with $\psi_0=\text{Id}_X$ such that $\psi^*\alpha_t=f_t\alpha_0$, for smooth functions $f_t$ on $X$ \cite{MS2}. This implies the following theorem of J.~W.~Gray \cite{Gr}. 
\begin{thm} Let $\xi_t$, with $t\in [0,1]$, be a smooth family of contact structures on a closed manifold $X$. Then there exists an isotopy $\psi_t$, $t\in [0,1]$, of orientation preserving self-diffeomorphisms of $X$ such that $\psi_t^*\xi_t=\xi_0$.  
\end{thm}
Because of this theorem we call contact structures $\xi, \xi'$ which can be deformed into each other by a smooth family of contact structures {\em isotopic}. We call almost contact structures {\em homotopic} if they can be connected by a smooth family of almost contact structures. The contact structures in an isotopy class or the almost contact structures in a homotopy class all have the same Chern classes. We can also consider (almost) contact structures $\xi,\xi'$ which are permuted by an orientation-preserving self-diffeomorphism $\psi$ of $X$, in the sense that $\psi^*\xi'=\xi$. 

\begin{defn}\label{defn inequiv almost cont} We call almost contact structures and contact structures on an oriented manifold $X$ {\em equivalent} if they can be made identical by a combination of deformations (homotopies and isotopies, respectively) and orientation preserving self-diffeomorphisms of $X$.
\end{defn}

The existence question for {\em almost contact structures} on 5-manifolds was settled by the following theorem of Gray \cite{Gr}.
\begin{thm} Let $X$ be a closed orientable 5-manifold. Then $X$ admits an almost contact structure if and only if $W_3(X)=0$.
\end{thm}
Here $W_3(X)\in H^3(X)$ is the third integral Stiefel-Whitney class, defined as the image of $w_2(X)$ under the Bockstein homomorphism.

The existence of {\em contact structures} on simply-connected 5-manifolds was proved by H.~Geiges \cite{Ge}. He also proved a classification theorem for almost contact structures on simply-connected 5-manifolds up to homotopy: 

\begin{thm}\label{Theorem geiges} Let $X$ be a simply-connected closed 5-manifold. 
\begin{enumerate}
\item Every class in $H^2(X)$ that reduces mod $2$ to $w_2(X)$ arises as the first Chern class of an almost contact structure. Two almost contact structures $\xi_0,\xi_1$ are homotopic if and only if $c_1(\xi_0)=c_1(\xi_1)$.
\item Every homotopy class of almost contact structures admits a contact structure.
\end{enumerate} 
\end{thm} 
A different proof for the existence of contact structures on simply-connected 5-manifolds can be found in \cite{Koe1,Koe3}. The fact that two almost contact structures are homotopic if they have the same first Chern class holds more generally for closed oriented 5-manifolds without 2-torsion in $H^2(X)$. For a proof see \cite[Theorem 8.18]{MHthesis}.

We want to prove the following theorem, which is a consequence of Barden's classification theorem.
\begin{thm}\label{thm on action of diff on 5-mfds}
Suppose that $X$ is a simply-connected closed oriented 5-manifold. Let $c,c'\in H^2(X)$ be classes with the same divisibility and whose mod 2 reduction is equal to $w_2(X)$. Then there exists an orientation preserving self-diffeomorphism $\phi\colon X\rightarrow X$ such that $\phi^*c'=c$. 
\end{thm}
Note that $H^2(X)$ is torsion free by the universal coefficient theorem, because $X$ is simply-connected. By {\em divisibility} we mean the maximal (non-negative) divisibility as an element in the free abelian group $H^2(X)$. The divisibility is zero if and only if the class is zero itself. The proof of the theorem uses the following lemma.
\begin{lem} \label{lemma div} Let $G$ be a finitely generated free abelian group of rank $n$. Suppose $\alpha \in \text{Hom}(G,\mathbb{Z})$ is indivisible. Then there exists a basis $e_1,\dotsc,e_n$ of $G$ such that $\alpha(e_1)=1$ and $\alpha(e_i)=0$ for $i>1$. 
\end{lem}
\begin{proof} The kernel of $\alpha$ is a free abelian subgroup of $G$ of rank $n-1$. Let $e_2,\dotsc,e_n$ be a basis of $\text{ker}\,\alpha$. The image of $\alpha$ in $\mathbb{Z}$ is a subgroup, hence of the form $m\mathbb{Z}$. Since $\alpha$ is indivisible we have $m=1$, so there exists an $e_1\in G$ such that $\alpha(e_1)=1$. 
The set $e_1,\dotsc,e_n$ is linearly independent. They also span $G$, because if $g\in G$ is some element then $\alpha(g-\alpha(g)e_1)=0$, hence $g=\alpha(g)e_1 +\sum_{i\geq 2}\lambda_ie_i$.
\qed\end{proof}
We can now prove Theorem \ref{thm on action of diff on 5-mfds}.

\begin{proof} By the universal coefficient theorem we have $H^2(X)\cong \text{Hom}(H_2(X),\mathbb{Z})$, since $X$ is simply-connected. Hence we can view $c,c'$ as homomorphisms on $H_2(X)$ with values in $\mathbb{Z}$. Let $p\colon \mathbb{Z}\rightarrow \mathbb{Z}_2$ denote mod $2$ reduction. The assumption on $c$ and $c'$ is equivalent to
\begin{equation*}
w_2(X)=p\circ c=p\circ c',
\end{equation*}
as homomorphisms on $H_2(X)$ with values in $\mathbb{Z}_2$. Since $c$ and $c'$ have the same divisibility we can write 
\begin{equation*}
c=k\alpha, \quad c'=k\alpha'
\end{equation*}
with $\alpha,\alpha' \in \text{Hom}(H_2(X),\mathbb{Z})$ indivisible. Let $H_2(X)=G\oplus \text{Tor}\,H_2(X)$ with $G$ free abelian. Since $c$ and $c'$ are homomorphisms to $\mathbb{Z}$ they vanish on $\text{Tor}\,H_2(X)$. By Lemma \ref{lemma div} there exist bases $e_1,\dotsc,e_n$ and $e_1',\dotsc,e_n'$ of $G$ such that 
\begin{equation*}
\alpha(e_1)=1=\alpha'(e_1'), \quad \alpha(e_k)=0=\alpha'(e_k') \quad \forall k>1.
\end{equation*}
Let $\theta$ be the group automorphism of $H_2(X)$ given by $\theta(e_k)=e_k'$ for all $k\geq 1$ and which is the identity on $\text{Tor}\,H_2(X)$. Then 
\begin{equation*}
(c'\circ\theta)(e_k)=c'(e_k')=c(e_k) \quad \forall k\geq 1.
\end{equation*}
Hence $c'\circ\theta=c$ on the free abelian subgroup $G$. This equality holds on all of $H_2(X)$ because $c$ and $c'$ vanish on the torsion subgroup. By the assumption above this implies that $w_2(X)\circ\theta =w_2(X)$. Moreover, since $\theta$ is the identity on $\text{Tor}\,H_2(X)$ it preserves the linking form. By Theorem \ref{Barden} the automorphism $\theta$ is induced by an orientation preserving self-diffeomorphism $\phi\colon X\rightarrow X$ such that $\phi_*=\theta$. We have
\begin{equation*}
c(\lambda)=c'(\phi_*\lambda) = (\phi^*c')(\lambda), \quad \text{for all $\lambda \in H_2(X)$}.
\end{equation*}
Hence $\phi^*c'=c$.
\qed\end{proof}

We get the following corollary for almost contact structures. 

\begin{cor}\label{Class 5-mfd almost cont}
Let $X$ be a simply-connected closed oriented 5-manifold. Then two almost contact structures $\xi_0$ and $\xi_1$ on $X$ are equivalent if and only if $c_1(\xi_0)$ and $c_1(\xi_1)$ have the same divisibility in integral cohomology.
\end{cor}    
One direction is clear because homotopies do not change the Chern class and self-diffeomorphisms of the manifold do not change the divisibility. The other direction follows from Theorem \ref{thm on action of diff on 5-mfds} and the first part of Theorem \ref{Theorem geiges}.

\begin{defn} For an almost contact structure $\xi$ on a simply-connected 5-manifold $X$ we denote by $d(\xi)$ the divisibility of $c_1(\xi)$ as a class in the free abelian group $H^2(X)$. 
\end{defn}

We call $d(\xi)$ the {\em level} of the almost contact structure $\xi$. By Corollary \ref{Class 5-mfd almost cont} almost contact structures and hence contact structures on a simply-connected 5-manifold $X$ naturally form a \textquotedblleft spectrum" consisting of levels which are indexed by the divisibility of the first Chern class. Two contact structures on $X$ are equivalent as almost contact structures if and only if they lie on the same level. By Lemma \ref{Contact spin} simply-connected spin 5-manifolds have only even levels and non-spin 5-manifolds only odd levels. In Section \ref{sect contact homology} we will use invariants from contact homology to investigate the \textquotedblleft fine-structure" of contact structures on each level in this spectrum. For instance, O.~van Koert \cite{Koe1} has shown that for many simply-connected 5-manifolds the lowest level, given by divisibility $0$, contains infinitely many inequivalent contact structures.

\section{Topology of circle bundles}\label{sect topology circle bundles}
In this section we collect some results on the topology of circle bundles. In particular, we determine which simply-connected closed 5-manifolds can arise as circle bundles over 4-manifolds. 
 
Let $M$ be a closed connected oriented $n$-manifold. For a second integral cohomology class $c$ on $M$ consider the map
\begin{equation*}
\langle c,-\rangle\colon H_2(M)\longrightarrow\mathbb{Z},
\end{equation*}
given by evaluation.
\begin{defn} We call the class $c$ {\em indivisible} if $\langle c,-\rangle$ is surjective.
\end{defn} 
Clearly, if the class $c$ is indivisible, then $c$ cannot be written as $c=ka$, with $k>1$ and $a\in H^2(M)$. By Poincar\'e duality it follows that a class $c\in H^2(M)$ is indivisible if and only if the map 
\begin{equation*}
c\,\cup\colon H^{n-2}(M)\longrightarrow H^n(M)\cong\mathbb{Z}
\end{equation*}
is surjective.

Suppose that $\pi\colon X\rightarrow M$ is the total space of an oriented circle bundle over $M$ with Euler class $e\in H^2(M)$. For the following proofs we will need two results which are probably well known but included here for completeness. The first result is related to the exact Gysin sequence \cite{MSt}:
\begin{equation*}
\ldots\stackrel{\pi^*}{\longrightarrow} H^k(X)\stackrel{\pi_*}{\longrightarrow} H^{k-1}(M) \stackrel{\cup e}{\longrightarrow} H^{k+1}(M)\stackrel{\pi^*}{\longrightarrow} H^{k+1}(X)\stackrel{\pi_*}{\longrightarrow}\ldots 
\end{equation*}
The homomorphism $\pi_*$ is called integration along the fibre and can be characterized in the following way.
\begin{lem}\label{int along fibre lemma} Integration along the fibre $\pi_*\colon H^{k+1}(X)\rightarrow H^k(M)$ is Poincar\'e dual to the map $\pi_*\colon H_{n-k}(X)\rightarrow H_{n-k}(M)$ induced by the projection.
\end{lem}
This follows because integration along the fibre is a so-called transfer or shriek map \cite{Br}. The second result is related to the long exact homotopy sequence associated to the fibration
\begin{equation*}
\ldots\longrightarrow \pi_2(M)\stackrel{\partial}{\longrightarrow} \pi_1(S^1) \longrightarrow \pi_1(X) \stackrel{\pi_*}{\longrightarrow} \pi_1(M)\longrightarrow 1.
\end{equation*}
\begin{lem}\label{lem relation partial sequence hurewicz euler} The map $\partial\colon \pi_2(M)\rightarrow \pi_1(S^1)\cong\mathbb{Z}$ in the long exact homotopy sequence for fibre bundles is given by
\begin{equation*}
\pi_2(M)\stackrel{h}{\longrightarrow}H_2(M)\stackrel{\langle e,-\rangle}{\longrightarrow}\mathbb{Z}
\end{equation*}
where $h$ denotes the Hurewicz homomorphism.
\end{lem}
This follows for example by considering the universal bunde $ES^1\rightarrow BS^1$.
\begin{lem} The Euler class $e$ is indivisible if and only if $\pi_*\colon H_1(X)\rightarrow H_1(M)$ is an isomorphism. Both statements are equivalent to the fibre $S^1\subset X$ being null-homologous.
\end{lem} 
\begin{proof}
Consider the following part of the Gysin sequence:
\begin{equation*}
\ldots\longrightarrow H^{n-2}(M) \stackrel{\cup e}{\longrightarrow} H^{n}(M)\longrightarrow H^{n}(X)\stackrel{\pi_*}{\longrightarrow} H^{n-1}(M)\longrightarrow 0.
\end{equation*}
This shows that $e$ is indivisible if and only if $\pi_*\colon H^n(X)\rightarrow H^{n-1}(M)$ is an isomorphism, in other words
\begin{equation*}
\pi_*\colon H_1(X)\longrightarrow H_1(M)
\end{equation*}
is an isomorphism.
The long exact homotopy sequence of the fibration $S^1\rightarrow X\rightarrow M$ induces by abelianization an exact sequence
\begin{equation*}
H_1(S^1)\longrightarrow H_1(X)\longrightarrow H_1(M)\longrightarrow 0. 
\end{equation*}
Hence we see that $e$ is indivisible if and only if the fibre $S^1\subset X$ is null-homologous. 
\qed\end{proof}
From the long exact homotopy sequence above we see that the fibre is {\em null-homotopic} if and only if $\partial\colon \pi_2(M)\rightarrow \pi_1(S^1)$ is surjective. By Lemma \ref{lem relation partial sequence hurewicz euler} this happens if and only if $\langle e,-\rangle$ is surjective on spherical classes. Both statements are equivalent to 
\begin{equation*}
\pi_*\colon \pi_1(X)\longrightarrow \pi_1(M)
\end{equation*}
being an isomorphism. 

\begin{lem} \label{both simply} $X$ is simply-connected if and only if $M$ is simply-connected and $e$ is indivisible. 
\end{lem}
\begin{proof} If $X$ is simply-connected the long exact homotopy sequence shows that $\pi_1(M)=1$ and $\partial\colon \pi_2(M)\rightarrow \pi_1(S^1)$ is surjective. Hence $M$ is simply-connected and the surjectivity of $\partial$ implies that $e$ is indivisible. Conversely, suppose that $M$ is simply-connected and $e$ is indivisible. Then the Hurewicz map $h\colon \pi_2(M)\rightarrow H_2(M)$ is an isomorphism and it follows that $\partial$ is surjective. The long exact homotopy sequence then implies the exact sequence $1\rightarrow \pi_1(X)\rightarrow 1$. Hence $\pi_1(X)=1$. 
\qed\end{proof}
The next lemma follows from the Gysin sequence.
\begin{lem} \label{pi* surj} 
Suppose the first Betti number of $M$ vanishes, $b_1(M)=0$. Then the map $\pi^*\colon H^2(M)\rightarrow H^2(X)$ is surjective with kernel $\mathbb{Z} e$. 
\end{lem}
Similarly we have:
\begin{lem}\label{im proj H2}
The image of the map $\pi_*\colon H_2(X)\rightarrow H_2(M)$ is the kernel of $\langle e,-\rangle$. If $H_1(M)=0$, then this map is injective.
\end{lem}

We now determine when the total space $X$ is spin.
\begin{lem}\label{X is spin} The total space $X$ is spin if and only if $w_2(M)\equiv \alpha e \bmod 2$ for some $\alpha\in\{0,1\}$, i.e.~if and only if $M$ is spin or $w_2(M)\equiv e \bmod 2$. 
\end{lem}
\begin{proof}
We claim that the following relation holds:
\begin{equation*}
w_2(X)=\pi^*w_2(M).
\end{equation*}
This follows because the tangent bundle of $X$ is given by $TX=\pi^*TM\oplus\mathbb{R}$ and the Whitney sum formula implies $w_2(TX)=w_2(\pi^*TM)\cup w_0(\mathbb{R}) = \pi^*w_2(TM)$. Hence $X$ is spin if and only if $w_2(M)$ is in the kernel of $\pi^*$. We consider the following part of the $\mathbb{Z}_2$-Gysin sequence:
\begin{equation*}
H^0(M;\mathbb{Z}_2) \stackrel{\cup \overline{e}}{\longrightarrow} H^2(M;\mathbb{Z}_2)\stackrel{\pi^*}{\longrightarrow} H^2(X;\mathbb{Z}_2),
\end{equation*}
where $\overline{e}$ denotes the mod 2 reduction of $e$. Hence the kernel of $\pi^*$ is $\{0,\overline{e}\}$. This implies the claim. 
\qed\end{proof}
We now specialize to the case where the dimension of $M$ is equal to 4. 
\begin{thm}\label{class 4-5 circle bundles} Let $X$ be a simply-connected closed oriented 5-manifold which is a circle bundle over a closed oriented connected 4-manifold $M$. Then $M$ is simply-connected and the Euler class $e$ is indivisible. Moreover, the integral homology and cohomology of $X$ are torsion free and given by:
\begin{itemize}
\item $H_0(X)\cong H_5(X)\cong\mathbb{Z}$
\item $H_1(X)\cong H_4(X)\cong 0$
\item $H_2(X)\cong H_3(X)\cong \mathbb{Z}^{b_2(M)-1}$. 
\end{itemize}
\end{thm}
\begin{proof} We only have to prove that the cohomology of $X$ is torsion free and the formula for $H_2(X)$. The cohomology groups $H^0(X),H^1(X)$ and $H^5(X)$ are always torsion free for an oriented 5-manifold $X$. We have the following part of the Gysin sequence:
\begin{equation*}
\ldots\longrightarrow H^3(M)\stackrel{\pi^*}{\longrightarrow} H^3(X)\stackrel{\pi_*}{\longrightarrow}H^2(M) \longrightarrow\ldots
\end{equation*}
By assumption $H^3(M)=0$. Therefore the homomorphism $\pi_*$ injects $H^3(X)$ into $H^2(M)$, which is torsion free by the assumption that $M$ is simply-connected and the universal coefficient theorem. Hence $H^3(X)$ is torsion free itself. It remains to consider $H^2(X)$ and $H^4(X)$. By the universal coefficient theorem and Poincar\'e duality $H^2(X)$ is torsion free if and only if $H_1(X)$ is torsion free, if and only if $H^4(X)$ is torsion free. Since $H_1(X)=0$, we see that $H^2(X)$ and $H^4(X)$ are torsion free. 

By Lemma \ref{pi* surj} we have $H^2(X)\cong H^2(M)/\mathbb{Z} e$. The cohomology group $H^2(M)$ is torsion free by the universal coefficient theorem. Since the class $e$ is indivisible we have $H^2(M)/\mathbb{Z} e\cong \mathbb{Z}^{b_2(M)-1}$. This implies the formula for $H_2(X)\cong H_3(X)$.  
\qed\end{proof}

With Proposition \ref{5-mfd tor free H2} we get the following corollary (this has also been proved in \cite{DL}).  
\begin{cor}\label{simpl conn 5-dim circle bundle} Let $M$ be a simply-connected closed oriented 4-manifold and $X$ the circle bundle over $M$ with indivisible Euler class $e$. Then $X$ is diffeomorphic to
\begin{enumerate} 
\item $X=\#(b_2(M)-1)S^2\times S^3$ if $X$ is spin
\item $X=\#(b_2(M)-2)S^2\times S^3\#S^2\tilde{\times}S^3$ if $X$ is not spin.
\end{enumerate} 
The first case occurs if and only if $M$ is spin or $w_2(M)\equiv e \bmod 2$.  
\end{cor}
Since every closed oriented 4-manifold is $\mathrm{spin}^c$ and hence $w_2(M)$ is the mod 2 reduction of an (indivisible) integral class, it follows as a corollary that every closed simply-connected 4-manifold $M$ is diffeomorphic to the quotient of a free and smooth $S^1$-action on $\#(b_2(M)-1)S^2\times S^3$.

\section{The Boothby-Wang construction}\label{Boothby Wang construct section}

We want to construct circle bundles over symplectic manifolds $M$ whose Euler class is represented by the symplectic form. Since the Euler class is an element of the integral cohomology group $H^2(M)$, the symplectic form has to represent an integral cohomology class in $H^2(M;\mathbb{R})$, i.e.~it has to lie in the image of the natural homomorphism
\begin{equation*}
H^2(M)\longrightarrow H^2(M;\mathbb{R})\stackrel{\cong}\longrightarrow H^2_{DR}(M).
\end{equation*}
The existence of such a symplectic form is guaranteed by the following argument (this argument is from \cite[Observation 4.3]{Go}): Let $(M,\omega)$ be a closed symplectic manifold. For every Riemannian metric on $M$ there exists a small $\epsilon$-ball $B_\epsilon$ around the origin in the space of harmonic 2-forms on $M$ such that every element in $\omega+B_\epsilon$ is symplectic. Since the set of classes in $H^2(M;\mathbb{R})$ represented by these elements is open, there exists a symplectic form which represents a rational cohomology class. By multiplication with a suitable rational number we can find a symplectic form which represents an integral class. If we want, we can choose the number such that the class is indivisible. Note also that all symplectic forms in $\omega+B_\epsilon$ can be connected to $\omega$ by a smooth path of symplectic forms. This implies that they all have the same Chern classes as $\omega$. 

We fix the following data:
\begin{enumerate}
\item A closed connected symplectic manifold $(M^{2n},\omega)$ with symplectic form $\omega$, representing an integral cohomology class in $H^2(M;\mathbb{R})$.
\item An integral lift $[\omega]_\mathbb{Z}\in H^2(M)$ of $[\omega]\in H^2_{DR}(M)$. 
\end{enumerate}
Let $\pi\colon X\rightarrow M$ be the principal circle bundle over $M$ with Euler class $e(X)=[\omega]_\mathbb{Z}$. By a theorem of Kobayashi \cite{Kob} we can choose a $U(1)$-connection $A$ on $X\rightarrow M$ whose curvature form $F$ is equal to $\tfrac{2\pi}{i}\omega$. Then $A$ is a 1-form on $X$ with values in $\mathfrak{u}(1)\cong i\mathbb{R}$ which is invariant under the $S^1$-action and there are the following relations, coming from the definition of a connection on a principal bundle: 
\begin{align*}
dA & = \pi^*F \\
A(R) & =2\pi i.
\end{align*}
Here $R$ denotes the fundamental vector field generated by the action of the element $2\pi i\in \mathfrak{u}(1)$. An orbit of $R$, topologically a fibre of $X$, has period 1.
\begin{prop} Define the real valued 1-form $\lambda = \frac{1}{2\pi i}A$ on $X$. Then $\lambda$ is a contact form\index{Contact form|(} on $X$ with 
\begin{align*}
d\lambda &= -\pi^*\omega \\
\lambda(R) & =  1.
\end{align*}
\end{prop}
\begin{proof}
We have the relations
\begin{align*}
dA &=-2\pi i \pi^*\omega \\
A(R)&=2\pi i.
\end{align*}\index{Connection on principal $S^1$-bundle|)}
This implies the corresponding relations for $\lambda$. The tangent bundle of $X$ splits as $TX\cong\mathbb{R}\oplus\pi^*TM$, where the trivial $\mathbb{R}$-summand is spanned by the vector field $R$. Fix a point of $X$ and choose a basis $(R,v_1,\dotsc,v_{2n})$ of its tangent space, where the $v_i$ form an oriented basis of the kernel of $\lambda$. Then
\begin{align*}
\lambda\wedge(d\lambda)^n(R,v_1,\dotsc,v_{2n}) & =  (d\lambda)^n(v_1,\dotsc,v_{2n}) \\
&= (-1)^n \omega^n(\pi_*v_1,\dotsc,\pi_*v_{2n}) \\
&\neq  0.
\end{align*}  
Hence $\lambda\wedge(d\lambda)^n$ is a volume form on $X$ and $\lambda$ is contact. 
\qed\end{proof}
\begin{rem} If we define the orientation on $X$ via the splitting $TX\cong\mathbb{R}\oplus\pi^*TM$, where the trivial $\mathbb{R}$-summand is oriented by $R$ and $TM$ by $\omega$, then $\lambda$ is a positive contact form if $n$ is even and negative otherwise. In particular, in the construction for a symplectic 4-manifold $M$ we get a positive contact form.
\end{rem}
\begin{defn}\label{def boothby-wang} The contact structure\index{Contact structure} $\xi$ on the closed oriented manifold $X^{2n+1}$, defined by the contact form\index{Contact form} $\lambda$ above, is called the {\em Boothby-Wang contact structure}\index{Boothby-Wang construction} associated to the symplectic manifold $(M,\omega)$\index{Symplectic structure!on a manifold|)}. Since $d\lambda (R)=0$, the Reeb vector field\index{Reeb vector field} of $\lambda$ is given by the vector field $R$ along the fibres.
\end{defn}
For the original construction see \cite{BoWa}.
\begin{prop}\label{inv of BW contact up to isotopy} If $\lambda'$ is another contact form, defined by a different connection $A'$ as above, then the associated contact structure $\xi'$ is isotopic\index{Contact structure!isotopic} to $\xi$.
\end{prop}
\begin{proof} The connection $A'$ is an $S^1$-invariant 1-form on $X$ with
\begin{align*}
dA'&=dA\\
A'(R)&=A(R).
\end{align*}
Hence $A'-A=\pi^*\alpha$ for some closed 1-form $\alpha$ on $M$. Define $A_t=A+\pi^*t\alpha$ for $t\in \mathbb{R}$. Then $A_t$ is a connection on $X$ with curvature $-2\pi i\omega$ for all $t$. Let $\lambda_t=\lambda+\pi^*(\frac{1}{2\pi i}t\alpha)$. Then $\lambda_t$ is a contact form on $X$ for all $t\in [0,1]$ with $\lambda_0=\lambda$ and $\lambda_1=\lambda'$. Therefore, $\xi$ and $\xi'$ are isotopic through the contact structures defined by $\lambda_t$. 
\qed\end{proof}
\index{Contact form|)}
The Chern classes of $\xi$ are given by the Chern classes associated to $\omega$ in the following way. 
\begin{lem}\label{Chern BW X} Let $X\rightarrow M$ be a Boothby-Wang fibration with contact structure $\xi$. Then $c_i(\xi)=\pi^*c_i(M,\omega)$ for all $i\geq 0$. The manifold $X$ is spin if and only if
\begin{equation*}
\text{$c_1(M,\omega)\equiv \alpha [\omega]_\mathbb{Z} \bmod 2$},
\end{equation*}
for some $\alpha\in\{0,1\}$. 
\end{lem} 
\begin{proof} Let $J$ be a compatible almost complex structure\index{Almost complex structure!complex structure on vector bundle} for $\omega$ on $M$. Then there exists a compatible complex structure $J'$ for $\xi$ on $X$ such that $\pi^*(TM,J)\cong (\xi,J')$ as complex vector bundles. The naturality of characteristic classes proves the first claim. The second claim follows from Lemma \ref{X is spin} and $c_1(M,\omega)\equiv w_2(M) \bmod 2$. 
\qed\end{proof}

\section{The construction for symplectic 4-manifolds}\label{sect construct 4-mfds}
We fix the following data:
\begin{enumerate}
\item A closed simply-connected symplectic 4-manifold $(M,\omega)$ with symplectic form $\omega$, representing an integral cohomology class in $H^2(M;\mathbb{R})$, given by the argument at the beginning of Section \ref{Boothby Wang construct section}. Since $H^2(M)$ is torsion free by the universal coefficient theorem, the class $[\omega]$ has a unique integral lift, denoted by $[\omega]_\mathbb{Z}\in H^2(M)$. We sometimes denote the integral lift also by $[\omega]$ or $\omega$. We assume that $[\omega]_\mathbb{Z}$ is indivisible. 
\item Let $\pi\colon X\rightarrow M$ be the principal $S^1$-bundle over $M$ with Euler class $e(X)=[\omega]_\mathbb{Z}$. Then $X$ is a closed simply-connected oriented 5-manifold with torsion-free homology by Theorem \ref{class 4-5 circle bundles}.
\item Let $\lambda$ be a Boothby-Wang contact form on $X$ with associated contact structure $\xi$. By Proposition \ref{inv of BW contact up to isotopy}, the contact structure $\xi$ does not depend on $\lambda$ up to isotopy. 
\end{enumerate}

By Corollary \ref{simpl conn 5-dim circle bundle} the 5-manifold $X$ is diffeomorphic to
\begin{itemize} 
\item $\#(b_2(M)-1)S^2\times S^3$ if $X$ is spin
\item $\#(b_2(M)-2)S^2\times S^3\#S^2\tilde{\times}S^3$ if $X$ is not spin.
\end{itemize} 

Hence the same abstract closed simply-connected 5-manifold $X$ with torsion free homology can be realized in several different ways as a Boothby-Wang fibration over different simply-connected symplectic 4-manifolds $M$ and therefore admits many, possibly non-equivalent, contact structures.

\begin{defn} The canonical class of the symplectic structure $\omega$ is defined as 
\begin{equation*}
K=-c_1(M,\omega)\in H^2(M).
\end{equation*}
We denote by $d(K)\geq 0$ the divisibility of $K$ in the free abelian group $H^2(M)$. Similarly, we define $d(\xi)$ to be the divisibility of $c_1(\xi)$.
\end{defn}

Note that $X$ is spin if and only if $d(\xi)$ is even by Lemma \ref{Contact spin}. With Corollary \ref{Class 5-mfd almost cont} we get:
\begin{prop}\label{prop dxi equivalence} Suppose that $(M',\omega')$ is another closed simply-connected symplectic 4-manifold with integral and indivisible symplectic form $\omega'$. Denote the associated Boothby-Wang total space by $(X',\xi')$. 
\begin{enumerate}
\item The simply-connected 5-manifolds $X$ and $X'$ are diffeomorphic if and only if $b_2(M)=b_2(M')$ and $d(\xi)\equiv d(\xi') \bmod 2$. 
\item If $X$ and $X'$ are diffeomorphic and $d(\xi)=d(\xi')$, then $\xi$ and $\xi'$ are equivalent as almost contact structures.
\end{enumerate}
\end{prop}  
The divisibility $d(\xi)$ can be calculated in the following way: By Lemma \ref{pi* surj} the bundle projection $\pi$ defines an isomorphism
\begin{equation*}
\pi^*\colon H^2(M)/\mathbb{Z}\omega\stackrel{\cong}{\longrightarrow}H^2(X),
\end{equation*}
and by Lemma \ref{Chern BW X} we have 
\begin{equation*}
\pi^*c_1(M)=c_1(\xi).
\end{equation*}
Let $[c_1(M)]$ denote the image of $c_1(M)$ in the quotient $H^2(M)/\mathbb{Z}\omega$, which is free abelian since $\omega$ is indivisible. We will use $\pi^*$ to identify 
\begin{align*}
H^2(X)&=H^2(M)/\mathbb{Z}\omega,\,\text{and}\\
c_1(\xi)&=[c_1(M)].
\end{align*}
Then $d(\xi)$ is also the divisibility of the class $[c_1(M)]$. If the second Betti number of $M$ is equal to $1$, then $H^2(X)=0$ and $d(\xi)=0$ trivially. For $b_2(M)>1$ we have:
\begin{lem}\label{lem calc dxi} The divisibility $d(\xi)$ is the maximal integer $d$ such that
\begin{equation*}
c_1(M)=dR+\gamma \omega
\end{equation*}
where $\gamma$ is some integer and $R\in H^2(M)$ not a multiple of $\omega$.
\end{lem}
An important fact is the following:
\begin{lem}\label{dxi multiple of dK} The integer $d(\xi)$ is always a multiple of $d(K)$.
\end{lem}
\begin{proof} We can write  $c_1(M)=d(K)W$ where $W$ is a class in $H^2(M)$. Then $[c_1(M)]=d(K)[W]$ in $H^2(M)/\mathbb{Z}\omega$. Since $d(\xi)$ is the divisibility of $[c_1(M)]$, the integer $d(\xi)$ has to be a multiple of $d(K)$.
\qed\end{proof}
Hence the possible levels of Boothby-Wang contact structures are restricted to the multiples of the divisibility of the canonical class.

\section{Contact homology}\label{sect contact homology}
In this section we consider invariants derived from symplectic field theory, introduced in \cite{EGH}. We only take into account the {\em classical contact homology} $H_*^{cont}(X,\xi)$ which is a graded supercommutative algebra, defined using rational holomorphic curves with one positive puncture and several negative punctures in the symplectization of the contact manifold. We use a variant of this theory for the so-called Morse-Bott case, described in \cite{Bou} and in \cite[Section 2.9.2]{EGH}. 

In general, the classical contact homology is the homology of a certain freely generated graded supercommutative algebra $\mathfrak{A}$ with a differential $\partial\colon \mathfrak{A}\rightarrow\mathfrak{A}$ that satisfies a Leibniz rule so that the homology becomes itself an algebra. The generators of $\mathfrak{A}$ correspond to periodic Reeb orbits of the contact form and the differential is associated to moduli spaces of holomorphic curves in the symplectization of the contact manifold which are asymptotic to these Reeb orbits. The degree of the generators is related to the Conley-Zehnder index of the corresponding closed Reeb orbit. The algebra $\mathfrak{A}$ is actually a family of algebras parametrized by $t\in H^*(X;\mathbb{R})$. Since each element of the family is preserved by the differential, we get a family of contact homology algebras which can be specialized at any parameter $t$. The homology algebra is an invariant of the contact structure that, up to isomorphism, does not depend on the choice of contact form.

We now describe the setup in our situation. We are going to associate to each Boothby-Wang fibration $\pi\colon X\rightarrow M$ as in the previous section a graded commutative algebra $\mathfrak{A}(X,M)$. Choose a basis $B_1,\dotsc,B_N$ of $H_2(X)$ where $N=b_2(X)=b_2(M)-1$ and let
\begin{equation*}
A_n=\pi_*B_n\in H_2(M),\,\,1\leq n\leq N.
\end{equation*} 
Note that
\begin{equation*}
c_1(B_n)=\langle c_1(\xi),B_n\rangle=\langle c_1(M),A_n\rangle =c_1(A_n).
\end{equation*}
Choose a class $A_0\in H_2(M)$ such that
\begin{equation*}
\omega(A_0)=1.
\end{equation*}
This is possible, because $\omega$ was assumed indivisible. The classes $A_0,A_1,\dotsc,A_N$ form a basis of $H_2(M)$ by Lemma \ref{im proj H2}. We consider variables 
\begin{align*}
z&=(z_1,\dotsc,z_N),\,\text{and}\\
q&=\{q_{k,i}\}_{k\in \mathbb{N}, \,0\leq i\leq a},
\end{align*}
where $a=b_2(M)+1$ and $\mathbb{N}$ denotes the set of positive integers. They have degrees defined by
\begin{align*}
\text{deg}(z_n)&=-2c_1(B_n) \\
\text{deg}(q_{k,i})&=\text{deg}\,\Delta_i-2+2c_1(A_0)k,
\end{align*}
where $\text{deg}\,\Delta_i$ is given by
\begin{equation*}
\text{deg}\,\Delta_i=\left\{\begin{array}{ll} 0 &\text{if}\,\,i=0\\ 
2&\text{if}\,\, i=1,\dotsc,b_2(M)\\ 4 &\text{if}\,\, i=b_2(M)+1.\end{array}\right.
\end{equation*}
In our situation the degree of all variables is even (hence the algebra we are going to define is truly commutative, not only supercommutative).
\begin{defn} We define the following algebras:
\begin{itemize}
\item $\mathfrak{L}(X)$ = $\mathbb{C}[H_2(X;\mathbb{Z})]$ = the graded commutative ring of Laurent polynomials in the variables $z$ with coefficients in $\mathbb{C}$.
\item $\mathfrak{A}(X,M)$ = $\bigoplus_{d\in\mathbb{Z}}\mathfrak{A}_d(X,M)$ = the graded commutative algebra of polynomials in the variables $q$ with coefficients in $\mathfrak{L}(X)$. Here $\mathfrak{A}_d(X,M)$ denotes the set of homogeneous elements of degree $d$. The degree of a polynomial is calculated using the definitions above.
\end{itemize}
A {\em homomorphism} $\phi$ of graded commutative algebras $\mathfrak{A},\mathfrak{A}'$ over $\mathfrak{L}(X)$
\begin{equation*}
\phi\colon \mathfrak{A}=\bigoplus_{d\in\mathbb{Z}}\mathfrak{A}_d\longrightarrow \mathfrak{A}'=\bigoplus_{d\in\mathbb{Z}}\mathfrak{A}'_d
\end{equation*}
is a homomorphism of rings which is the identity on $\mathfrak{L}(X)$ and such that $\phi(\mathfrak{A}_d)\subset \mathfrak{A}'_d$ for all $d\in\mathbb{Z}$.
\end{defn}
\begin{lem}The following statements hold: 
\begin{enumerate}
\item Up to isomorphism, the ring $\mathfrak{L}(X)$ does not depend on the choice of basis $B_1,\dotsc,B_N$ for $H_2(X)$. 
\item For fixed $\mathfrak{L}(X)$, the algebra $\mathfrak{A}(X,M)$ does not depend, up to isomorphism over $\mathfrak{L}(X)$, on the choice of the class $A_0\in H_2(M)$ as above.
\end{enumerate}
\end{lem}
\begin{proof} Let $\overline{B}_1,\dotsc,\overline{B}_N$ be another basis of $H_2(X)$ and $\overline{\mathfrak{L}(X)}$ the associated ring, generated by variables $\overline{z}$. Then there exists a matrix
\begin{equation*}
(\beta_{mn})\in GL(N,\mathbb{Z}),
\end{equation*}
with $1\leq m,n\leq N$, such that
\begin{equation*}
\overline{B}_m=\sum_{n=1}^N\beta_{mn}B_n.
\end{equation*}
Define a homomorphism $\phi\colon \overline{\mathfrak{L}(X)}\rightarrow\mathfrak{L}(X)$ via
\begin{equation*}
\overline{z}_m\mapsto\prod_{n=1}^N z_n^{\beta_{mn}},
\end{equation*}
for all $1\leq m\leq N$. Then $\phi$ preserves degrees and is an isomorphism, since the matrix $(\beta_{mn})$ is invertible.

Let $\overline{A}_0$ be another element in $H_2(M)$ such that $\omega(\overline{A}_0)=1$ and $\overline{\mathfrak{A}(X,M)}$ the associated algebra, generated by variables $\overline{q}$. Then there exists a vector
\begin{equation*}
(\alpha_n)\in \mathbb{Z}^N,
\end{equation*}
with $1\leq n\leq N$, such that
\begin{equation*}
\overline{A}_0=A_0+\sum_{n=1}^N\alpha_n A_n.
\end{equation*}
Define a homomorphism $\psi\colon \overline{\mathfrak{A}(X,M)}\rightarrow \mathfrak{A}(X,M)$ via
\begin{equation*}
\overline{q}_{k,i}\mapsto q_{k,i} \prod_{n=1}^N z_n^{-k\alpha_n},\quad k\in \mathbb{N}, \,0\leq i\leq a,
\end{equation*}
and which is the identity on $\mathfrak{L}(X)$. Then $\psi$ preserves degrees and is invertible.
\qed\end{proof}
\begin{defn} We choose a class $A_0\in H_2(M)$ with $\omega(A_0)=1$ and denote $c_1(A_0)$ by $\Delta$. Hence the degrees of the variables $q_{k,i}$ are equal to
\begin{equation*}
\text{deg}(q_{k,i}) =\text{deg}\,\Delta_i-2+2\Delta k.
\end{equation*}
\end{defn}
The integer $\Delta$ has the following properties.
\begin{lem}\label{deltagamma} The following relations hold:
\begin{enumerate} 
\item Let $c_1(M)=d(\xi)R+\gamma \omega$ for some class $R\in H^2(M)$ and integer $\gamma\in \mathbb{Z}$ as in Lemma \ref{lem calc dxi}. Then $\Delta \equiv\gamma \bmod d(\xi)$.
\item The greatest common divisor $\gcd(\Delta,d(\xi))$ is equal to $d(K)$. In particular, if $d(\xi)=0$, then $\Delta=d(K)$.\footnote{We use the convention that $\gcd(0,0)=0$.}
\end{enumerate}
\end{lem}
\begin{proof} Part (a) follows if we evaluate both sides on $A_0$. To prove part (b), the integer $d(K)$ divides $d(\xi)$ by Lemma \ref{dxi multiple of dK} and it divides $c_1(M)$, hence also $\Delta$. On the other hand, there exists a homology class $B\in H_2(M)$ such that $d(K)= c_1(M)(B)$. By part (a)
\begin{equation*}
d(K)=d(\xi)R(B)+\gamma\omega(B)
\end{equation*}
and $\gamma\equiv \Delta \bmod d(\xi)$. Hence there exist integers $x,y\in\mathbb{Z}$ such that $d(K)=xd(\xi)+y\Delta$. This proves the claim.
\qed\end{proof}

We are interested in the algebra $\mathfrak{A}(X,M)$ because of the following result, described in \cite[Proposition 2.9.1]{EGH}: 
\begin{thm}\label{conhom} For a Boothby-Wang fibration $X\rightarrow M$ as above, the Morse-Bott contact homology $H_*^{cont}(X,\xi)$ specialized at $t=0$ is isomorphic to $\mathfrak{A}(X,M)$.
\end{thm}
If two Boothby-Wang contact structures $\xi$ and $\xi'$ on $X$ are equivalent, then their contact homologies are isomorphic. We now make the following assumptions:
\begin{enumerate}
\item The simply-connected 5-manifold $X$ can be realized as the Boothby-Wang total space over another closed simply-connected symplectic 4-manifold $(M',\omega')$, where $\omega'$ represents an integral and indivisible class. This implies in particular that $b_2(M')=b_2(M)$ and both are equal to $a-1$. Denote the canonical class of $(M',\omega')$ by $K'$ and its divisibility by $d(K')$
\item We assume that $\xi$ and $\xi'$ are contact structures on the same level and therefore both are equivalent as almost contact structures. We set $d=d(\xi')=d(\xi)$. 
\item Let $\mathfrak{A}(X,M')$ denote the associated algebra over $\mathfrak{L}(X)$, generated by variables $\{q'_{l,j}\}$, with $l\in\mathbb{N}, 0\leq j\leq a$. We set $\Delta'=c_1(A_0')$ for a class $A_0'$ with $\omega'(A_0')=1$. 
\end{enumerate}
The following is the main theorem in this section:
\begin{thm}\label{isom cont hom} The algebras $\mathfrak{A}(X,M)$ and $\mathfrak{A}(X,M')$ are isomorphic over $\mathfrak{L}(X)$ if and only if one of the following three conditions is satisfied:
\begin{itemize}
\item $d\geq 1$ and both $d(K),d(K')\leq 3$
\item $d=0$ and $d(K)=d(K')$
\item $d\geq 4$ and $d(K)=d(K')\geq 4$.
\end{itemize} 
\end{thm} 
This shows that the isomorphism type of the contact homology for Boothby-Wang contact structures on the same level is strongly related to the divisibility of the canonical class of the symplectic structure. The proof of this theorem is done in several steps.

\begin{defn} Suppose that $d\geq 1$. For $0\leq b< d$ denote by $Q_b$ the set of generators $\{q_{k,i}\}$ with
\begin{equation*}
\text{$\text{deg}(q_{k,i})\equiv 2b \bmod 2d$}.
\end{equation*}
The set of all generators is the disjoint union of the sets $Q_b$. Similarly denote by $Q'_b$ the set of generators $\{q'_{l,j}\}$ with
\begin{equation*}
\text{$\text{deg}(q'_{l,j})\equiv 2b \bmod 2d$}.
\end{equation*}
\end{defn}
The following lemma shows that there is a relation between the cardinality of the set $Q_b$ of generators and the divisibility of the canonical class of the symplectic structure.   
\begin{lem} \label{Qb empty infinite} Assume that $d\geq 1$. Then the set $Q_b$ is infinite if $d(K)$ divides one of the integers $b-1,b,b+1$ and empty otherwise.
\end{lem}
\begin{proof} Suppose $d(K)=\gcd(\Delta,d)$ divides one of the integers $b+\epsilon$, with $\epsilon\in\{-1,0,1\}$. Then the equation
\begin{equation*}
b=-\epsilon+\Delta k+d\alpha
\end{equation*}
has infinitely many solutions $k\geq 1$ with $\alpha\in\mathbb{Z}$. Choose an integer $0\leq i\leq a$ with $\text{deg}\,\Delta_i-2=-2\epsilon$. Then
\begin{equation*}
\text{$\text{deg}(q_{k,i})=-2\epsilon+2\Delta k\equiv 2b \bmod 2d$}
\end{equation*}
for infinitely many $k\geq 1$. Hence these $q_{k,i}$ are all in $Q_b$. 

Conversely, suppose that $d(K)$ does not divide any of the integers $b+\epsilon$, with $\epsilon\in\{-1,0,1\}$. Suppose that $Q_b$ contains an element $q_{l,j}$. We have $\text{deg}(q_{l,j})=-2\epsilon+2\Delta l$ for some $\epsilon\in\{-1,0,1\}$. By assumption,
\begin{equation*}
\text{deg}(q_{l,j})=-2\epsilon+2\Delta l=2b-2d\alpha,
\end{equation*}
for some $\alpha\in\mathbb{Z}$. This implies
\begin{equation*}
b+\epsilon=\Delta l+d\alpha.
\end{equation*}
This is impossible, since $d(K)$ divides the right side, but not the left side.
\qed\end{proof}
\begin{ex}\label{example dK leq 123 then all Qb infinite} Suppose that $d\geq 1$. If $d(K)\in \{1,2,3\}$, then Lemma \ref{Qb empty infinite} implies that $Q_b$ is infinite for all $b=0,\dotsc, d-1$. If $d(K)\geq 4$ (and hence $d\geq 4$ as well), then at least one of the $Q_b$ is empty, e.g.~$Q_2$ is always empty in this case.
\end{ex}

Lemma \ref{Qb empty infinite} implies the following relation between the cardinalities of the set of generators $Q_b$ and $Q'_b$ and the divisibilities of the canonical classes of the symplectic 4-manifolds $M$ and $M'$.  
\begin{lem}\label{Qb and Q'b} Assume that $d\geq 4$ and at least one of the numbers $d(K),d(K')$ is $\geq 4$. Then the following two statements are equivalent:
\begin{enumerate}
\item There exists an integer $0\leq b<d$ such that $Q_b$ and $Q'_b$ do not have the same cardinality (i.e. one of them is empty and the other infinite).
\item $d(K)\neq d(K')$.
\end{enumerate}
\end{lem}
\begin{proof} Suppose that $d(K)=d(K')$. By Lemma \ref{Qb empty infinite}, the sets $Q_b$ and $Q'_b$ have the same cardinality for all $0\leq b<d$. Conversely, suppose that $d(K)\neq d(K')$; without loss of generality $d(K)<d(K')$. If $d(K)\in \{1,2,3\}$ let $b=2$. Then $Q_2$ is infinite, while $Q'_2$ is empty (since $d(K')\geq 4$ by assumption). If $d(K)\geq 4$ let $b=d(K)-1\geq 3$. Then $d(K)$ divides $b+1$, but $d(K')$ does not divide any of the integers $b-1,b,b+1$. Hence $Q_b$ is infinite and $Q'_b$ empty.
\qed\end{proof}
Using Lemma \ref{Qb and Q'b}, we can prove the following. 
\begin{lem}\label{Lemma if algebras isomorphic then dK=dK'} Suppose that either (i) $d=0$ or (ii) $d>0$ and at least one of the numbers $d(K),d(K')$ is $\geq 4$. If the algebras $\mathfrak{A}(X,M)$ and $\mathfrak{A}(X,M')$ are isomorphic, then $d(K)=d(K')$.
\end{lem}
This implies one direction of Theorem \ref{isom cont hom}.
\begin{proof} Suppose that $d=0$ and that there is an isomorphism $\phi\colon \mathfrak{A}(X,M)\rightarrow \mathfrak{A}(X,M')$. Note that all elements in $\mathfrak{L}(X)$ have degree zero. Depending on the sign of $\Delta$ we consider the element of highest or lowest degree in $\mathfrak{A}(X,M)$, and similarly in $\mathfrak{A}(X,M')$. Since $\phi$ has to preserve degree, this implies $\Delta=\Delta'$ and hence
\begin{equation*}
d(K)=\gcd(\Delta,0)=\Delta=\Delta'=\gcd(\Delta',0)=d(K').
\end{equation*}
Now assume that $d>0$ and at least one of $d(K),d(K')$ is $\geq 4$. By Lemma \ref{dxi multiple of dK}, the integer $d$ is at least $4$. Suppose that $d(K)\neq d(K')$ and there exists an isomorphism $\phi\colon \mathfrak{A}(X,M)\rightarrow \mathfrak{A}(X,M')$. 

By Lemma \ref{Qb and Q'b}, there exists an integer $0\leq b<d$ such that $Q_b$ and $Q'_b$ have different cardinality. Without loss of generality, we may assume that $Q_b$ is empty and $Q'_b$ infinite (otherwise we consider $\phi^{-1}$). Let $q'_{r,s}$ be a generator in $Q'_b$. Then $q'_{r,s}$ is a polynomial in the images 
\begin{equation*}
\{\phi(q_{k,i})\}_{k\in\mathbb{N}, 0\leq i\leq a},
\end{equation*}
with coefficients in $\mathfrak{L}(X)$ and we can write
\begin{equation*}
q'_{r,s}=f(\phi(q_{k_1,i_1}),\dotsc,\phi(q_{k_v,i_v}))\in \mathfrak{L}(X)[\phi(q_{k_1,i_1}),\dotsc,\phi(q_{k_v,i_v})].
\end{equation*}
The images $\phi(q_{k,i})$ are themselves polynomials in the variables $\{q'_{l,j}\}$ with coefficients in $\mathfrak{L}(X)$. Expressed as a polynomial in the variables $\{q'_{l,j}\}$, at least one of the images $\phi(q_{k_w,i_w})$, $1\leq w\leq v$, must contain a summand of the form $\alpha q'_{r,s}$ with $\alpha\in\mathfrak{L}(X)$ non-zero. Since $\phi$ preserves degrees, the element $\phi(q_{k_w,i_w})$ is homogeneous of degree
\begin{equation*}
\text{$\text{deg}(\phi(q_{k_w,i_w}))=\text{deg}(\alpha q'_{r,s})\equiv \text{deg}(q'_{r,s})\equiv 2b \bmod 2d$}.
\end{equation*}
This implies $\text{deg}(q_{k_w,i_w})\equiv 2b \bmod 2d$, hence $q_{k_w,i_w}\in Q_b$. This is impossible, since $Q_b$ is empty.  
\qed\end{proof} 
The other direction of Theorem \ref{isom cont hom} follows from the next lemma.
\begin{lem}\label{Lemma dK=dK' implies algebras isomorphic} Suppose that either (i) $d(K)=d(K')$ or (ii) both numbers $d(K),d(K')$ are $\leq 3$ and $d\neq 0$. Then the algebras $\mathfrak{A}(X,M)$ and $\mathfrak{A}(X,M')$ are isomorphic over $\mathfrak{L}(X)$.
\end{lem}
\begin{proof} We can choose a basis $B_1,\dotsc,B_N$ of $H_2(X)$ such that
\begin{align*}
c_1(B_1)&=d(\xi)=d\\
c_1(B_n)&=0,\quad \text{for all $2\leq n\leq N$}.
\end{align*} 
Choose elements $A_0\in H_2(M)$ and $A'_0\in H_2(M')$ on which the symplectic forms evaluate to one and set
\begin{equation*}
\Delta=c_1(A_0),\,\,\Delta'=c_1(A'_0).
\end{equation*}
We will use these bases to define the algebras $\mathfrak{A}(X,M)$ and $\mathfrak{A}(X,M')$. Suppose that $d=0$ and $d(K)=d(K')$. Then 
\begin{align*}
\Delta&=\gcd(\Delta,0)=d(K)\\
\Delta'&=\gcd(\Delta',0)=d(K').
\end{align*}
This implies $\text{deg}(q_{k,i})=\text{deg}(q'_{k,i})$ for all $k\in\mathbb{N}, 0\leq i\leq a$. Hence the map
\begin{equation*}
q_{k,i}\mapsto q'_{k,i}, \quad k\in\mathbb{N},\,0\leq i\leq a,
\end{equation*}
induces a degree preserving isomorphism $\phi\colon \mathfrak{A}(X,M)\rightarrow \mathfrak{A}(X,M')$.

Suppose $d\geq 1$. Under our assumptions, the sets $Q_b$ and $Q'_b$ have the same cardinality for each $0\leq b<d$, cf.~Lemma \ref{Qb and Q'b} and Example \ref{example dK leq 123 then all Qb infinite}. Hence there exists a bijection
\begin{equation*}
\psi\colon \mathbb{N}\times\{0,\dotsc,a\}\longrightarrow \mathbb{N}\times\{0,\dotsc,a\}, (k,i)\mapsto \psi(k,i),
\end{equation*}
such that
\begin{equation*}
\text{$\text{deg}(q_{k,i})\equiv \text{deg}(q'_{\psi(k,i)}) \bmod 2d$}.
\end{equation*}
Since $z_1$ has degree $-2d$, there exists for each $(k,i)\in\mathbb{N}\times\{0,\dotsc,a\}$ an integer $\alpha(k,i)\in\mathbb{Z}$, such that
\begin{equation*}
\text{deg}(q_{k,i})= \text{deg}\left({z_1}^{\alpha(k,i)}q'_{\psi(k,i)}\right).
\end{equation*}
The map
\begin{equation*}
q_{k,i}\mapsto {z_1}^{\alpha(k,i)}q'_{\psi(k,i)}, \quad k\in\mathbb{N},\,0\leq i\leq a,
\end{equation*}
therefore induces a well-defined isomorphism $\phi\colon \mathfrak{A}(X,M)\rightarrow \mathfrak{A}(X,M')$ over $\mathfrak{L}(X)$, preserving degrees. 
\qed\end{proof}

Using Theorem \ref{isom cont hom} and Proposition \ref{prop dxi equivalence} we get the following corollary. The part concerning equivalent contact structures follows because equivalent contact structures have isomorphic contact homologies.  

\begin{cor}\label{cor different contact str} Let $X$ be a closed simply-connected 5-manifold which can be realized in two different ways as a Boothby-Wang fibration over closed simply-connected symplectic 4-manifolds $(M_1,\omega_1)$ and $(M_2,\omega_2)$, whose symplectic forms represent integral and indivisible classes:
\begin{equation*}
\xymatrix{&X \ar[dl]_{\pi_1}\ar[dr]^{\pi_2}&\\
          (M_1,\omega_1)&&(M_2,\omega_2)}
\end{equation*}
Denote the associated Boothby-Wang contact structures on $X$ by $\xi_1$ and $\xi_2$ and the canonical classes of the symplectic structures by $K_1$ and $K_2$. Let $d(\xi_i)$ denote the divisibility of the first Chern class of $\xi_i$ and $d(K_i)$ the divisibility of $K_i$. Then:
\begin{itemize}
\item The almost contact structures underlying $\xi_1$ and $\xi_2$ are equivalent if and only if $d(\xi_1)=d(\xi_2)$.
\end{itemize}
Suppose that $\xi_1$ and $\xi_2$ are equivalent as contact structures.
\begin{itemize}
\item If $d(\xi_1)=d(\xi_2)=0$, then $d(K_1)=d(K_2)$.
\item If $d(\xi_1)=d(\xi_2)\neq 0$, then either both $d(K_1),d(K_2)\leq 3$ or $d(K_1)=d(K_2)\geq 4$.
\end{itemize}
\end{cor}

\section{Applications}\label{sect applications}

In order to apply Corollary \ref{cor different contact str} it is useful to have as many contact structures on different levels of $X$ as possible. By Lemma \ref{dxi multiple of dK}, the level is always a multiple of the divisibility of the canonical class. We first want to show that one can perturb a single symplectic form $\omega$ on a given simply-connected 4-manifold $M$ without changing the canonical class $K$, so that the induced Boothby-Wang contact structures realize all levels which are non-zero multiples of the divisibility $d(K)$.

For the following lemma, recall that a symplectic 4-manifold is called {\em minimal} if it does not contain an embedded symplectic sphere $S$ of self-intersection $-1$. If $S$ is such a sphere and $K$ the canonical class, then the intersection number $K\cdot S$ is equal to $-1$ by the adjunction formula. Hence if the divisibility $d(K)$ is at least two, then $M$ is minimal.
\begin{lem}\label{inflation canonical class} Let $(M,\omega)$ be a minimal closed symplectic 4-manifold with $b_2^+(M)>1$ and canonical class $K$. Then every class in $H^2(M;\mathbb{R})$ of the form $[\omega]+tK$ for a real number $t\geq 0$ can be represented by a symplectic form.
\end{lem} 
\begin{proof} Note that the canonical class $K$ is a Seiberg-Witten basic class. Since $M$ is assumed minimal, Proposition 3.3 and the argument in Corollary 3.4 in \cite{HK} show that $K$ is represented by a disjoint collection of embedded symplectic surfaces in $M$, all of which have non-negative self-intersection. The inflation procedure \cite{LalMcD,McD}, which can be done on each of the surfaces separately and with the same parameter $t\geq 0$, shows that $[\omega]+tK$ is represented by a symplectic form on $M$.
\qed\end{proof}

We can now prove:
\begin{thm}\label{Cieliebak} Let $M$ be a closed minimal simply-connected 4-manifold such that $b_2^+(M)>1$ and $\omega$ a symplectic form on $M$. Denote the canonical class of $\omega$ by $K$ and let $m\geq 1$ be an arbitrary integer. Then there exists a symplectic form $\omega'$ on $M$, deformation equivalent to $\omega$ and representing an integral and indivisible class, such that the first Chern class of the associated Boothby-Wang contact structure $\xi'$ has divisibility $d(\xi')=md(K)$. 
\end{thm}
\begin{proof} Let $k=d(K)$. We can assume that $\omega$ is integral and choose a basis for $H^2(M;\mathbb{Z})$ such that
\begin{align*}
K&=k(1,0,\dotsc,0)\\
\omega&=(\omega_1,\omega_2,0,\dotsc,0).
\end{align*}
By a deformation we can assume that $\omega$ is not parallel to $K$, hence $\omega_2\neq 0$. We can also assume that $\omega_1$ is negative while $\omega_2$ is positive: Consider the change of basis vectors
\begin{align*}
(1,0,0,\dotsc,0)&\mapsto (1,0,0,\dotsc,0)\\
(0,1,0,\dotsc,0)&\mapsto (q,\pm 1,0,\dotsc,0),
\end{align*}
where $q$ is some integer. Then the expression of $\omega$ in the new basis is
\begin{equation*}
(\omega_1\mp q\omega_2,\pm \omega_2,0,\dotsc,0).
\end{equation*}
Hence if $q$ is large enough, has the correct sign and the $\pm$ sign is chosen correctly, the claim follows.

Suppose that $\sigma\in H^2(M;\mathbb{Z})$ is an indivisible class of the form
\begin{equation*}
\sigma=(\sigma_1,\sigma_2,0,\dotsc,0)
\end{equation*} 
which can be represented by a symplectic form, also denoted by $\sigma$, with canonical class $K$. Let $\zeta$ denote the contact structure induced on the Boothby-Wang total space by $\sigma$. We claim that the divisibility $d(\zeta)$ is given by
\begin{equation*}
d(\zeta)=k|\sigma_2|.
\end{equation*}
To prove this we write $K=-c_1(M)=rR+\gamma \sigma$ as in Lemma \ref{lem calc dxi}, where $R$ is a class of the form $R=(R_1,R_2,0...,0)$. Then $k-\gamma\sigma_1$ and $\gamma \sigma_2$ are divisible by $r$. This implies that $r$ divides $k\sigma_2$. Conversely note that by assumption $\sigma_1,\sigma_2$ are coprime. Let $R_1,R_2$ be integers with
\begin{equation*}
1=\sigma_2R_1-\sigma_1R_2
\end{equation*}
and define
\begin{equation*}
\gamma=-kR_2.
\end{equation*}
Then we can write
\begin{equation*}
K=k\sigma_2R-kR_2\sigma.
\end{equation*}
This proves the claim about $d(\zeta)$. 
 
Suppose that $m\geq 1$. By multiplying the expression for $\omega$ with the positive number $\frac{m}{\omega_2}$ we see that the (rational) class
\begin{equation*}
(\alpha, m, 0,\dotsc,0),\quad \alpha=\omega_1\tfrac{m}{\omega_2},
\end{equation*}
is represented by a symplectic form. Note that $\alpha<0$. By the inflation trick in Lemma \ref{inflation canonical class} with parameter $t=\frac{1}{k}(1-\alpha)$ it follows that
\begin{align*}
\omega'&=(\alpha, m, 0,\dotsc,0)+(1-\alpha,0,\dotsc,0)\\
&=(1,m,0,\dotsc,0)
\end{align*}
is represented by a symplectic form $\omega'$. The class $\omega'$ is indivisible and has canonical class $K$. Let $\xi'$ denote the induced Boothby-Wang contact structure. By our calculation above we have $d(\xi')=mk$.
\qed\end{proof}

\begin{defn} For integers $d\geq 4$ and $r\geq 2$ consider the following set:
\begin{equation*}
\Gamma(r,d)=\left\{k\in \mathbb{N}\Biggl\lvert\begin{array}{l} \text{$k\geq 1$, $k$ divides $d$ and there exists a simply-connected minimal}\\ \text{symplectic 4-manifold $(M,\omega)$ with $b_2(M)=r$ and $b_2^+(M)>1$}\\
\text{whose canonical class $K$ has divisibility $d(K)=k$.} \end{array}\right\}
\end{equation*}
We define an integer $Q(r,d)$ by counting the number of elements of $\Gamma(r,d)$ as follows: If there are integers $k\in \Gamma(r,d)$ with $k\leq 3$ we count only one of them and we count each integer $k\geq 4$ once.
\end{defn}
\begin{ex}
Suppose that for some integers $r,d$ we have
\begin{equation*}
\Gamma(r,d)=\{1,3,4,7,12\}.
\end{equation*}
Then $Q(r,d)=4$. If we have
\begin{equation*}
\Gamma(r,d)=\{1,2,3,6,12,19,27\},
\end{equation*}
then $Q(r,d)=5$.
\end{ex}
The numbers $Q(r,d)$ are connected to the geography of simply-connected symplectic 4-manifolds with divisible canonical class. The following lemma relates knowledge about the numbers $Q(r,d)$ to the existence of inequivalent contact structures on simply-connected 5-manifolds. Here we make essential use of Corollary \ref{cor different contact str} and Theorem \ref{Cieliebak}.
\begin{lem}\label{relation divis number contact} Let $d\geq 4$ and $r\geq 2$ be integers. Suppose that either 
\begin{itemize}
\item $d$ is odd and $X$ the simply-connected 5-manifold $\#(r-2)S^2\times S^3\#S^2\tilde{\times}S^3$, or
\item $d$ is even and $X$ the simply-connected 5-manifold $\#(r-1)S^2\times S^3$. 
\end{itemize}
In both cases, there exist at least $Q(r,d)$ many inequivalent contact structures on the level $d$ on $X$. 
\end{lem}
\begin{proof} Recall that a spin (non-spin) simply-connected 5-manifold has only even (odd) levels. Suppose that $d\geq 4$ is an integer and $(M,\omega)$ a simply-connected minimal symplectic 4-manifold with $b_2(M)=r$ and $b_2^+(M)>1$ whose canonical class has divisibility $k=d(K)$ dividing $d$. We can write $d=mk$. By Theorem \ref{Cieliebak} there exists a symplectic structure $\omega'$ on $M$ that induces on the Boothby-Wang total space $X$ with $b_2(X)=r-1$ a contact structure with $d(\xi)=d$. Since the symplectic form $\omega'$ is deformation equivalent to $\omega$ the canonical class $K$ remains unchanged. By Corollary \ref{cor different contact str} the contact structures on the same non-zero level $d$ on $X$ coming from symplectic 4-manifolds with different divisibilities of their canonical classes, at most one divisibility less than $4$, are pairwise inequivalent.  
\qed\end{proof}
We have the following purely number theoretic definition.
\begin{defn} Let $d\geq 4$ be an integer. Consider the number of divisors greater or equal to four of $d$. Then $N(d)$ is this number plus one. If $d$ is even, consider the number of odd divisors greater or equal to four of $d$. Then $N'(d)$ is this number plus one.
\end{defn}
\begin{ex}
The different divisors of $60$ are
\begin{equation*}
1,2,3,4,5,6,10,12,15,20,30,60.
\end{equation*}
Hence $N(60)=10$ and $N'(60)=3$.
\end{ex}
The following lemma gives a bound on the maximal number of inequivalent contact structures that can be distinguished with our method. The proof uses some well-known properties of 4-manifolds that can be found for example in \cite{GS}.
\begin{lem}\label{lemma bounds on Q} Let $d\geq 4$ and $r\geq 2$ be integers. Then there are the following upper bounds for $Q(r,d)$.
\begin{enumerate}
\item For any $r$ we have $Q(r,d)\leq N(d)$.
\item If $d$ is even and $r$ is not congruent to $2 \bmod 4$, then $Q(r,d)\leq N'(d)$.
\end{enumerate}
\end{lem}
\begin{proof} The first statement is clear by the definitions. For the second statement, suppose that $M$ is a simply-connected symplectic spin 4-manifold. Note that $b_2^-=b_2^+-\sigma$, hence $b_2(M)=2b_2^+(M)-\sigma(M)$. Since $M$ is spin, the signature $\sigma(M)$ is divisible by $16$ according to Rohlin's theorem. This implies that $b_2(M)$ is congruent to $2 \bmod 4$, because $b_2^+(M)$ is odd for a simply-connected symplectic 4-manifold. Hence if $r$ is {\em not} congruent to $2 \bmod 4$ then there does not exist a simply-connected symplectic spin 4-manifold $M$ with second Betti number $r$. Since the divisibility of the canonical class of a non-spin symplectic 4-manifold is odd, this implies that in case (b) all numbers in the set $\Gamma(r,d)$ are odd.
\qed\end{proof} 

To calculate some of the numbers $Q(r,d)$ we can use the geography work in \cite{MHgeo}. Recall the following definition:
\begin{defn}A {\em homotopy elliptic surface} $M$ is a closed simply-connected 4-manifold homeomorphic to a relatively minimal simply-connected elliptic surface.
\end{defn}
Every relatively minimal simply-connected elliptic surface is diffeomorphic to a surface of the form $E(m)_{p,q}$ with $p$ and $q$ coprime integers. Here $E(m)$ denotes the (up to diffeomorphism unique) relatively minimal simply-connected elliptic surface without multiple fibres and Euler characteristic equal to $12m$ and $E(m)_{p,q}$ is obtained by two logarithmic transformations with indices $p$ and $q$, see \cite{GS}.  By definition, homotopy elliptic surfaces have topological invariants
\begin{equation*}
c_1^2(M)=0, \chi_h(M)=m, b_2(M)=12m-2, b_2^+(M)=2m-1,
\end{equation*}
for some integer $m\geq 1$. There are many constructions of exotic homotopy elliptic surfaces which are not diffeomorphic to elliptic surfaces. In \cite[Theorem 15]{MHgeo} we proved the following:
\begin{thm}\label{geo of hom elliptic surface}
Let $m$ and $k$ be positive integers. If $m$ is odd, assume that $k$ is odd also. Then there exists a symplectic homotopy elliptic surface $M$ with $\chi_h(M)=m$ whose canonical class $K$ has divisibility $k$.
\end{thm}
\begin{rem}
We will only use this theorem for $m\geq 2$. It follows from the construction that these symplectic homotopy elliptic surfaces are minimal. This is clear if $k\geq 2$ but holds also if $k=1$, because the manifolds are constructed using fibre sums and there is a way to determine when such a manifold is minimal \cite{Ush}. 
\end{rem}
This implies the following proposition about some of the numbers $Q(r,d)$:
\begin{prop}\label{Q for homotopy elliptic surfaces}
Let $d\geq 4$ be an integer.
\begin{enumerate}
\item If $d$ is odd and $n\geq 2$, then $Q(12n-2, d)=N(d)$.
\item Suppose that $d$ is even. If $n\geq 1$ then $Q(24n-2, d)=N(d)$ and if $n\geq 2$ then $Q(24n-14, d)\geq N'(d)$.
\end{enumerate}
\end{prop} 
\begin{proof} For part (a), let $r=12n-2$ and suppose that $d\geq 4$ is odd. To prove the claim we first find for every divisor $k\geq 4$ of the integer $d$ a simply-connected symplectic 4-manifold $M$ with $b_2=r$ and $b_2^+>1$ whose canonical class has divisibility equal to $k$: Since $d$ is odd, the integer $k$ is odd as well. By Theorem \ref{geo of hom elliptic surface} there exists a symplectic homotopy elliptic surface $M$ with $b_2=r$, $b_2^+\geq 3$ and $d(K)=k$. There also exists a minimal symplectic homotopy elliptic surface with the same invariants and $d(K)=1$. This implies the claim.

To prove part (b), suppose that $d\geq 4$ is even and let $r=24n-2$. Then for every divisor $k\geq 4$ of $d$ there exists by Theorem \ref{geo of hom elliptic surface} a symplectic homotopy elliptic surface $M$ with $b_2=r$, $b_2^+\geq 3$ and $d(K)=k$. Suppose that $r=24n-14=12(2n-1)-2$. Then for every odd divisor $k\geq 4$ of $d$ there exists by Theorem \ref{geo of hom elliptic surface} a symplectic homotopy elliptic surface $M$ with $b_2=r$, $b_2^+\geq 3$ and $d(K)=k$. In both cases there exists a minimal symplectic homotopy elliptic surface with the same invariants and $d(K)=1$. This proves the second claim.
\qed\end{proof}
As a corollary we get the following result about the existence of inequivalent contact structures in the same equivalence class of almost contact structures:
\begin{cor}\label{cor on existence of contact} Let $n\geq 1$ be an arbitrary integer.
\begin{enumerate}
\item On every odd level $d\geq 5$ the 5-manifold $\#(12n-4)S^2\times S^3\#S^2\tilde{\times}S^3$ admits at least $N(d)$ inequivalent contact structures.
\item On every even level $d\geq 4$ the 5-manifold $\#(24n-3)S^2\times S^3$ admits at least $N(d)$ inequivalent contact structures.
\item On every even level $d\geq 4$ the 5-manifold $\#(24n-15)S^2\times S^3$ admits at least $N'(d)$ inequivalent contact structures.
\end{enumerate}
\end{cor}
\begin{proof} This follows from Proposition \ref{Q for homotopy elliptic surfaces} and Lemma \ref{relation divis number contact} in all cases except for the first and last case with $n=1$. In these cases we choose as $M$ a Dolgachev surface $E(1)_{p,q}$, where $p$ and $q$ are coprime positive integers and $b_2(M)=10$ and $b_2^+(M)=1$. The canonical class of a Dolgachev surface is given by $K=(pq-p-q)f$ where $f$ is an indivisible class. For every odd integer $k$ we choose $p=2$ and $q=k+2$.  It follows that we can realize all odd numbers $k$ as $d(K)$ for $b_2(M)=10$. Since the canonical class of these Dolgachev surfaces is a positive multiple of the class represented by a symplectic torus of self-intersection zero, given by one of the multiple fibres, the proofs of Lemma \ref{inflation canonical class} and Theorem \ref{Cieliebak} also work in this case even though $b_2^+=1$.
\qed\end{proof}
Note that $N(d)\geq 2$ for all $d\geq 4$, hence in the first two cases we always get at least two inequivalent contact structures. In a similar way we can use other geography results from \cite {MHgeo} to find more inequivalent contact structures on the same level on simply-connected 5-manifolds $X$ of the form $\#rS^2\times S^3$ and $\#rS^2\times S^3\#S^2\tilde{\times}S^3$. 

\begin{rem} In \cite{Ler} E.~Lerman considered on $M=S^2\times S^2$ the symplectic forms
\begin{equation*}
\omega_{a,b}=a\omega_1+b\omega_2,
\end{equation*}
where $\omega_i$ is the pull-back by the projection onto the $i$-th factor of the standard area form with integral one on $S^2$ and $a>b\geq 1$ are coprime integers. Since the symplectic class is indivisible and $M$ is spin it follows that the Boothby-Wang total space $X$ is diffeomorphic to $S^2\times S^3$. The first Chern class of $M$ is $c_1(M)=2[\omega_1]+2[\omega_2]$ for all $a,b$ and the level of the induced contact structure $\xi_{a,b}$ on $X$ is
\begin{equation*}
d(\xi_{a,b})=2(a-b).
\end{equation*}
Hence if the difference $a-b$ is fixed, we get contact structures on the same level. Lerman asks if $\xi_{a',b'}$ and $\xi_{a,b}$ on the same level are equivalent as contact structures. Unfortunately, we cannot answer this question with Corollary \ref{cor different contact str} because these contact structures all arise from symplectic forms with $d(K)=2$.
\end{rem}

\section*{Acknowledgements} The content of this article is part of the author's Ph.D.~thesis, submitted in May 2008 at the University of Munich. I would like to thank D.~Kotschick for supervising the thesis and K.~Cieliebak, T.~Jentsch and O.~van Koert for helpful comments. I would also like to thank the referee for suggesting ways to improve the article. Finally, I am grateful to the {\em Studienstiftung des deutschen Volkes} and the {\em Deutsche Forschungsgemeinschaft (DFG)} for financial support.

\bigskip
\bigskip

\end{document}